\documentclass[a4paper,12pt]{article}
\usepackage{a4wide}
\usepackage{amsmath}
\usepackage{amssymb}
\usepackage{amsthm}
\usepackage{latexsym}
\usepackage{graphicx}
\usepackage[english]{babel}
\usepackage{makeidx}
\usepackage{mathtools}
              
\newtheorem{dfn} [subsection]{Definition}
\newtheorem{obs} [subsection]{Remark}

\newtheorem{prop}[subsection]{Proposition}

\newtheorem{teor}[subsection]{Theorem}
\newtheorem{lema}[subsection]{Lemma}

\def\Cons{\operatorname{Cons}}

\def\SS{\operatorname{S}}
\def\CC{\operatorname{C}}

\def\SL{\operatorname{SL}}

\def\Ar{\operatorname{Ar}}
\def\Ch{\operatorname{Ch}}

\def\Gal{\operatorname{Gal}}
\def\Hol{\operatorname{Hol}}
\def\Lin{\operatorname{Lin}}

\def\Irr{\operatorname{Irr}}

\def\Ker{\operatorname{Ker}}
\numberwithin{equation}{section}
\usepackage{titletoc}
\dottedcontents{section}[2.5em]{}{2em}{1pc}
\begin{document}
\selectlanguage{english}
\frenchspacing

\title{A note on the monomial characters of a wreath product of groups}
\author{Mircea Cimpoea\c s$^1$}
\date{}

\maketitle

\begin{abstract}
Given a quasi-monomial, respectively an almost monomial, group $A$ and a cyclic group $C$ of prime order $p>0$, we 
show that the wreath product $W=A\wr C$ is quasi-monomial (respectively almost monomial), if certain technical
conditions hold.


\end{abstract}

\footnotetext[1]{ \emph{Mircea Cimpoea\c s}, University Politehnica of Bucharest, Faculty of
Applied Sciences, 
Bucharest, 060042, Romania and Simion Stoilow Institute of Mathematics, Research unit 5, P.O.Box 1-764,
Bucharest 014700, Romania, E-mail: mircea.cimpoeas@upb.ro,\;mircea.cimpoeas@imar.ro}

\section*{Introduction}

We recall that a finite group $G$ is monomial, if any irreducible character $\chi$ of $G$ is induced by a linear character 
$\lambda$ of a subgroup $H$ of $G$. More generally, a finite group $G$ is quasi-monomial, if for any irreducible character $\chi$ of
$G$, there exist a subgroup $H$ of $G$ and a linear character $\lambda$ of $H$ such that $\lambda^G=d\chi$ for a positive integer $d$.

The notion of \emph{almost monomial groups}, which is a loose generalization of (quasi)-monomial groups,
was introduced by F.\ Nicolae in a recent paper \cite{monat}, in connection with the study of Artin L-functions
associated to a Galois extension $K/\mathbb Q$. A finite group $G$ is called \emph{almost monomial}, if for any two 
irreducible characters $\chi\neq \phi$ of $G$,
there exist a subgroup $H\leqslant G$ and a linear character $\lambda$ of $H$, such that $\chi$ is a constituent
of the induced character $\lambda^G$ and $\phi$ is not. In the first section, we present the main definitions and a motivation
for our study.

A key ingredient in the proof of Dade's theorem (\cite[Theorem 9.7]{isaac2}), i.e. any solvable group is isomorphic
to subgroup of some monomial group, is the following result: If $A$ is a monomial group and $C$ is
cyclic of prime order $p > 0$, then the wreath product $W = A \wr C$ is monomial. The main purpose of this note
is to generalize this result this in the framework of quasi-monomial and almost monomial groups: In Proposition \ref{bobo}
we prove that if $A$ is quasi-monomial and $C$ is cyclic of order $p>0$, then $W=A\wr C$ is quasi-monomial if a special
condition holds. In Theorem \ref{main}, we prove that if $A$ is quasi-monomial and $C$ is cyclic of order $p>0$, the $W=A\wr C$ 
is almost-monomial if a technical condition holds.

\newpage
\section{Preliminaries}

Let $G$ be a finite group. We denote $\Ch(G)$ the set of characters associated to the linear 
representations of the group $G$ over the complex field. We denote $\Irr(G)=\{\chi_1,\ldots,\chi_r\}$ 
the set of irreducible characters of $G$. It is well known that any character $\chi\in \Ch(G)$ can be uniquely written 
as linear combination $\chi=a_1\chi_1+a_2\chi_2+\cdots+a_r\chi_r$ where $a_i$'s are nonnegative integers and
not all of them are zero. The set of constituents of $\chi$ is $\Cons(\chi)=\{\chi_i\;:\;\langle \chi,\chi_i \rangle = a_i>0\}$,
where $\langle \chi,\phi \rangle = \sum_{g\in G} \chi(g)\overline{\phi(g)}$.

A character $\lambda$ of $G$ is called linear, if $\lambda(1)=1$. Obviously, the linear
characters are irreducible. We denote $\Lin(G)$, the set of linear characters of $G$. Note that a group $G$ is commutative if and
only if $\Lin(G)=\Irr(G)$.

If $H\leqslant G$ is a subgroup and $\chi$ is a character of $G$, then the restriction of $\chi$ to $H$, denoted by $\chi_H$,
is a character of $H$. If $\theta$ is a character of $H$, then 
$$\theta^G(g):= \frac{1}{|H|}\sum_{x\in G} \theta^0(xgx^{-1}),\;\text{ for all }g\in G,$$
where $\theta^0(x)=\theta(x)$, for all $x\in H$, and $\theta^0(x)=0$, for all $x\in G\setminus H$, is a character of $G$, which is
called the character induced by $\theta$ on $G$.

\begin{dfn}
Let $G$ be a finite group. Let $\chi$ a character of $G$.
\begin{enumerate}
\item[(1)] $\chi$ is called \emph{monomial} if there exists a subgroup $H\leqslant G$ and a linear character $\lambda$ of $H$ such that $\lambda^G=\chi$.
\item[(2)] $\chi$ is called \emph{quasi-monomial} if there exists a subgroup $H\leqslant G$ and a linear character $\lambda$ of $H$ such that $\lambda^G=d\cdot \chi$,
      for some positive integer $d$.
\end{enumerate}
\end{dfn}

We mention that, according to a theorem of Taketa, see \cite{taketa} or \cite[Theorem A]{dorn}, all monomial groups are solvable.
In general, the converse is not true, the smallest example being the group $\SL_2(\mathbb F_3)$; see \cite[p. 67]{isaac2}.

\begin{dfn}
Let $G$ be a finite group.
\begin{enumerate}
\item[(1)] The group $G$ is called monomial if all the irreducible characters of $G$ are monomial.
\item[(2)] The group $G$ is called quasi-monomial if all the irreducible characters of $G$ are quasi-monomial.
\end{enumerate}
\end{dfn}

It is not known if there are quasi-monomial groups which are not monomial. However, it was proved in \cite{konig} that a special class of 
quasi-monomial groups are solvable.
F.\ Nicolae \cite{monat} introduced the following generalization of (quasi)-monomial groups:

\begin{dfn}
A finite group $G$ is called \emph{almost monomial}, if for any two irreducible characters $\chi\neq \phi$ of $G$,
there exist a subgroup $H\leqslant G$ and a linear character $\lambda$ of $H$, such that 
$\chi\in \Cons(\lambda^G)$ and $\phi\notin\Cons(\lambda^G)$.
\end{dfn}

Obviously, any (quasi)-monomial group $G$ is also almost monomial; the converse however is false. For instance, the already mentioned group $\SL_2(\mathbb F_3)$
is almost monomial, but is not (quasi)-monomial. We recall the main results from \cite{lucrare}, regarding the almost monomial groups:

\begin{teor}\label{22} We have that:
\begin{enumerate}
\item[(1)] The symmetric group $\SS_n$ is almost monomial for any $n\geq 1$.(\cite[Theorem 2.1]{lucrare})
\item[(2)] If $G$ is almost monomial and $N\unlhd G$ is a normal subgroup, then $G/N$ is almost monomial.(\cite[Theorem 2.2]{lucrare})
\item[(3)] $G_1$ and $G_2$ are almost monomial if and only if $G_1\times G_2$ is almost monomial. (\cite[Theorem 2.3]{lucrare})
\end{enumerate}
\end{teor}

The main motivation in studying (quasi)-monomial and almost monomial groups is given by their connection with the theory of Artin L-functions.
Let $K/\mathbb Q$ be a finite Galois extension. 
For the character $\chi$ of a representation of the Galois group $G:=\Gal(K/\mathbb Q)$
on a finite dimensional complex vector space, let $L(s,\chi):=L(s,\chi,K/\mathbb Q)$ be the corresponding Artin L-function 
(\cite[P.296]{artin2}). 

Artin conjectured that $L(s,\chi)$ is holomorphic in $\mathbb C\setminus \{1\}$ and $s=1$ is a simple pole. 
If the character $\chi$ is (quasi)-monomial, then $L(s,\lambda^G)$ is holomorphic in $\mathbb C\setminus \{1\}$.
In particular, if a group $G$ is (quasi)-monomial, Artin Conjecture holds.

Using the fact that any character of $G$ can be written as a $\mathbb Z$-linear combination of monomial characters, 
Brauer \cite{brauer} proved that $L(s,\chi)$ is meromorphic in $\mathbb C$, of order $1$.

Let $\chi_1,\ldots,\chi_r$ be the irreducible characters of $G$, $f_1=L(s, \chi_1),\ldots,f_r=L(s,\chi_r)$ the corresponding Artin L-functions.
In \cite{forum} we proved that $f_1,\ldots,f_r$ are algebraically independent over the field of meromorphic functions of order $<1$.
We consider
$$\Ar:=\{f_1^{k_1}\cdot\ldots\cdot f_r^{k_r}\mid k_1\geq 0,\ldots,k_r\geq 0\}$$
the multiplicative semigroup of all  L-functions.
Now, let $s_0\in\mathbb C\setminus\{1\}$. We denote $\Hol(s_0)$ the semigroup of Artin $L$-functions holomorphic at $s_0$.
Obviously, we have the inclusion $$\Hol(s_0)\subset \Ar.$$
The main result of \cite{monat} is the following:

\begin{teor}\label{25}
If $G$ is almost monomial, then the following are equivalent:
\begin{enumerate}
\item[(1)] Artin's conjecture is true: $\Hol(s_0)=\Ar.$
\item[(2)] The semigroup $\Hol(s_0)$ is factorial. 
\end{enumerate}
\end{teor}

The main result of \cite{lucrare} is the following:

\begin{teor}\label{26}
If $G$ is almost monomial and $s_0$ is not a common zero for any two distinct L-functions $f_k$ and $f_l$ then all Artin L-functions of $K/\mathbb Q$ are holomorphic at $s_0$.  
\end{teor}

\section{Main results}

\begin{dfn}
Let $A$ be a nontrivial finite group and let $r\geq 1$ be an integer. Let $B$ be the direct product of $r$ copies of $A$, i.e.
$$B:=A\times A\times \cdots \times A = \{(a_1,a_2,\ldots,a_r)\;:\;a_j\in A,\;1\leq j\leq r\}.$$
We consider the automorphism
$$f:A\to A,\;f(a_1,\ldots,a_r):=(a_r,a_1,\ldots,a_{r-1}).$$
We consider the subgroup $T:=\langle f\rangle \subset Aut(B)$, which is cyclic of order $r$.
The \emph{wreath product} of $A$ by a cyclic group $C$ of order $r$, denoted by $W:=A\wr C$,
is the semidirect product of $B$ by $T$.
\end{dfn}

\begin{obs}\rm\label{obso}
Using the notations from the above definition, let $\chi\in \Irr(W)$ and assume that $\theta=\chi_B\in\Irr(B)$. 
It follows that $\theta$ is $T$-invariant in $W$ and
thus $\theta=\varphi\times \varphi \times \cdots \times \varphi$, where $\varphi\in \Irr(A)$.
\end{obs}

We recall the following result:

\begin{prop}(\cite[Lemma 9.5]{isaac2})\label{boboo}
If $C$ is a cyclic group of prime order $p>0$ and $A$ is monomial,
then $W=A\wr C$ is monomial. 
\end{prop}

A main ingredient in the proof is the following lemma:

\begin{lema}(\cite[Lemma 1.8]{isaac2})\label{lemaa}
Let $G$ be a finite group. Let $N\unlhd G$, where $|G:N|$ is prime and let $\theta\in \Irr(N)$. 
Then either $\theta$ has exactly $|G:N|$ conjugates in $G$ or else, for any $\chi\in \Irr(G)$ such that $\langle \theta, \chi_N \rangle>0$, we have $\chi_N=\theta$.
\end{lema}

\begin{lema}\label{prod}
Let $G_1,G_2$ be two quasi-monomial groups. Then $G_1\times G_2$ is quasi-monomial.
\end{lema}

\begin{proof}
It is well known that $\Irr(G_1\times G_2)=\Irr(G_1)\times \Irr(G_2)$. 

Let $\chi=(\chi_1,\chi_2)\in \Irr(G_1\times G_2)$.
Since $G_1$ and $G_2$ are quasi-monomial, it follows that there exists some subgroups $H_1\leqslant G_1$, $H_2\leqslant G_2$
some liniear characters $\lambda_1$ of $H_1$, $\lambda_2$ of $H_2$, and some positive integers $d_1$, $d_2$ such that
$ \lambda_1^{G_1}=d_1\chi_1 \text{ and } \lambda_2^{G_2}=d_2\chi_2.$
It follows that 
$$(\lambda_1 \times \lambda_2)^{G_1\times G_2} = \lambda_1^{G_1} \times \lambda_2^{G_2} = d_1\chi_1 \times d_2\chi_2 = (d_1d_2) \chi_1\times\chi_2 = (d_1d_2)\chi.$$
Therefore, the group $G_1\times G_2$ is quasi-monomial.
\end{proof}


Considering Remark \ref{obso}, we prove the following slight generalization of \ref{boboo}:

\begin{prop}\label{bobo}
Let $C$ be a cyclic group of prime order $p>0$ and let $A$ be a quasi-monomial group.
Let $W=A\wr C$ be the wreath product of $A$ by $C$. We assume that for any $\chi\in \Irr(W)$
with $\chi_B = \varphi\times \varphi \times \cdots \times \varphi \in\Irr(B)$, the character $\varphi\in \Irr(A)$ is monomial.

Then $W=A\wr C$ is quasi-monomial.
\end{prop}

\begin{proof}
Let $\chi\in \Irr(G)$. Let $B$ the base of the wreath product $W$, so $|W:B|=p$ is prime. 
Let $\theta\in \Irr(B)$ with $\langle \theta, \chi_B \rangle>0$. From \ref{lemaa} it follows
that either i) $\chi=\theta^G$, either ii) $\chi_B=\theta$.

Assume that $\chi=\theta^G$. Since $A$ is quasi-monomial, according to Lemma \ref{prod}, the group $B=A\times A\times \cdots \times A$ is also quasi-monomial.
Since $\theta\in\Irr(B)$, it follows that there exists $H\leqslant B$, $\lambda\in \Lin(H)$ and $d>0$ such that $\theta=d\lambda^B$.
If $\chi=\theta^G$, then $\chi=d\lambda^G$, thus $\chi$ is quasi-monomial.

We can assume that $\chi_B=\theta$ and thus $\theta$ is invariant in $W$. Then $\theta=\varphi \times \varphi \times \cdots \times \varphi$, 
where $\varphi \in \Irr(A)$. By our assumption, there exist $H\leqslant A$, $\lambda\in \Lin(A)$ such that $\lambda^A = \varphi$. The rest of
the proof is identical to the last part of the proof of \cite[Lemma 9.5]{isaac2}, hence we skip it.
\end{proof}


Our main result is the analog of Proposition \ref{bobo}, in the framework of almost monomial groups, with 
some technical condition added:

\begin{teor}\label{main}
Let $A$ be an almost monomial group and let $C$ be a cyclic group of prime order $p>0$. Let $W=A \wr C$ be
the wreath product of $A$ with $C$. Let $B\unlhd W$ be the base group of the wreath product $W$.

Assume that for any $\chi\in \Irr(W)$ with $\chi_B\in\Irr(B)$ and any non trivial character $\beta$ of $W/B$,
there exists a subgroup $H\leqslant W$ and a linear character $\lambda$ of $H$ such that $\chi\in \Cons(\lambda^W)$
and $\beta\chi \notin \Cons(\lambda^W)$.

Then $W=A \wr C$ is almost monomial.
\end{teor}

\begin{proof}
Let $\chi\neq \phi\in \Irr(W)$. Let $\theta$ be a constituent of
$\chi_B$. Since $|W:B|=p$ is prime, according to Lemma \ref{lemaa}, it follows that either i) $\chi=\theta^W$ or ii) $\chi_B=\theta$.

Assume $\theta=\theta_1\times \cdots \times \theta_p$, where $\theta_j\in \Irr(A)$, $1\leq j\leq p$.
\begin{enumerate}
 \item[(i)] If $\chi=\theta^W$, then $\chi_B$ is the sum of conjugates of $\theta$ in $W$, that is
$$\chi_B=\theta_1\times \theta_2 \times \cdots \times \theta_p + \theta_p\times \theta_1\times \cdots \times \theta_{p-1}+\cdots+\theta_2\times \cdots\times \theta_p\times \theta_1,\; |\{\theta_1,\ldots,\theta_p\}|\geq 2.$$
We denote $(12\ldots p)$ the cyclic permutation $(12\cdot p)(j):=\begin{cases} j+1,& j<p,\\ 1,& j=p \end{cases}$. Also, we denote 
$C_p:=\langle (12\cdots p )\rangle$ the cyclic subgroup generated by $(12\ldots p)$ in $\Sigma_p$, the group of permutations of order $p$. For $\sigma\in C_p$, we let
$\sigma(\theta_1\times \cdots \times \theta_p):=\theta_{\sigma(1)}\times \cdots \times \theta_{\sigma(p)}$.
With these notations, we have 
$\chi_B=\sum_{\sigma\in C_p}\sigma(\theta).$
 \item[(ii)] If $\chi_B=\theta$, then $\theta$ is invariant in $W$, hence $\theta=\theta_1\times\cdots\times \theta_1$, where $\theta_1 \in \Irr(A)$.
\end{enumerate}
Analogously, if $\eta$ is a constituent of $\phi_B$, then either a) $\eta=\phi^W$ or b) $\phi_B=\eta$. 

Assume $\eta=\eta_1\times \cdots \times \eta_p$, where $\eta_j\in \Irr(A)$, $1\leq j\leq p$.
\begin{enumerate}
 \item[(a)] $\phi=\eta^W$ and $\phi_B = \eta_1\times \eta_2 \times \cdots \times \eta_p + \eta_p\times \eta_1\times \cdots \times \eta_{p-1}+\cdots+\eta_2\times \cdots\times \eta_p\times \eta_1$,\; $|\{\eta_1,\ldots,\eta_p\}|\geq 2$. 
            In other words, $\phi_B=\sum_{\sigma\in C_p}\sigma(\eta)$.
 \item[(b)] $\phi_B=\eta=\eta_1\times\eta_1\times\cdots \times \eta_1$, where $\eta_1\in \Irr(A)$.
\end{enumerate}
We have to consider several cases:
\begin{enumerate}
\item[(i.a)] Assume there exists $1\leq j\leq p$ such that $\eta_j\notin \{\theta_1,\ldots,\theta_p\}$.
             Since $A$ is almost monomial, it follows that for any $1\leq i\leq p$, there exists a subgroup $H_i\leqslant A$ and $\lambda_i$ a linear character of 
						 $H_i$ such that $\langle \lambda_i^A,\theta_i \rangle > 0$
             and $\langle \lambda_i^A,\eta_j \rangle = 0$. Let $H:=H_1\times\cdots\times H_p$ and $\lambda:=\lambda_1\times \cdots\times\lambda_p$. 
						 Note that $\lambda$ is linear.
             Since $\theta\in\Cons(\chi_B)$, it follows that 
             $$\langle \lambda^W,\chi \rangle = \langle \lambda^B,\chi_B \rangle \geq \langle \lambda^B,\theta \rangle = \prod_{i=1}^p 
						  \langle \lambda_i^A,\theta_i \rangle > 0.$$
						 						
             Since, by (a), we have 
						 $\phi_B=\sum_{\sigma\in \langle (12\cdots p )\rangle} \eta_{\sigma(1)}\times \eta_{\sigma(2)}\times \cdots \times \eta_{\sigma(p)},$
						  it follows that
             $$\langle \lambda^W,\phi \rangle = \langle \lambda^B,\phi_B \rangle = \sum_{\sigma\in \langle(12\cdots p)\rangle} 
						 \prod_{i=1}^p \langle \lambda_i^A,\eta_{\sigma(i)} \rangle  = 0.$$

             Now, assume $\{\eta_1,\ldots,\eta_p\}\subseteq \{\theta_1,\ldots,\theta_p\}$.
             We claim that $\Cons(\chi_B)\cap\Cons(\phi_B)=\emptyset$. Indeed, if $\Cons(\chi_B)\cap \Cons(\phi_B)\neq\emptyset$, then there
             exists $\sigma\in C_p$ such that $\eta=\sigma(\theta)$. It follows that $\phi=\eta^W=(\sigma(\theta))^W=\theta^W=\chi$, a contradiction.
             
             Since $\Cons(\chi_B)\cap \Cons(\phi_B)=\emptyset$, it follows that 
             for any permutation $\sigma \in \langle (12\ldots p)\rangle \subset S_p$, we have
             that $$\theta_1\times \theta_2 \times \cdots \times \theta_p \neq 
						 \eta_{\sigma(1)}\times \eta_{\sigma(2)} \times \cdots \times \eta_{\sigma(p)} .$$
						 We use the convention $\theta_{j+p}=\theta_j$, $\eta_{j+p}=\eta_j$, for any $1\leq j\leq p$.

						 Let $1\leq i_1\leq p$ such that $\theta_{i_1}\neq \eta_{i_1}$. Let $1\leq j_1\leq p$ be the largest
						 integer with the property $\theta_{i_1}\neq \eta_{i_1},\;\theta_{i_1}\neq \eta_{i_1+1},\;\ldots,
						 \theta_{i_1}\neq \eta_{i_1+j_1-1}.$
						  If $j_1=p$, then we stop. 

Assume that $j_1<p$. Since $\theta_{i_1}=\eta_{i_1+j_1}$, we can choose				
							an index $i_2\neq i_1$ such that $\theta_{i_2}\neq \eta_{i_2+j_1}$.
							Let $1\leq j_2\leq n-j_1$ be the largest integer such that
							$\theta_{i_1}\neq \eta_{i_1+j_1} \text{ or } \theta_{i_2}\neq \eta_{i_2+j_1},\; 
							\theta_{i_1}\neq \eta_{i_1+j_1+1} \text{ or } \theta_{i_2}\neq \eta_{i_2+j_1+1},\;
							\theta_{i_1}\neq \eta_{i_1+j_1+j_2-1} \text{ or } \theta_{i_2}\neq \eta_{i_2+j_1+j_2-1}.$
							
							If $j_1+j_2=p$, then we stop. 

                                                         If $j_1+j_2<p$, as $\theta_{i_1} = \eta_{i_1+j_1+j_2}$ and
                                                         $\theta_{i_2}=\eta_{i_2+j_1+j_2}$, then we can choose $i_3\notin \{i_1,i_2\}$ such that
							$\theta_{i_3}\neq \eta_{i_2+j_1+j_2}$ etc.

                                                        Using the same procedure, we construct two sequences
							of positive integers $i_1,i_2,\ldots,i_s$ and $j_1,j_2,\ldots,j_s$ such that $|\{i_1,\ldots,i_s\}|=s\leq p$,
							 $j_1+\ldots+j_s=p$ and for any $1\leq \ell\leq p$, there exists some $1\leq t\leq s$ and such that 
							 $\theta_{i_t}\neq \eta_{i_t+\ell-1}$.
							
							For $1\leq t \leq s$, we choose $H_{i_t}$ a subgroup in $A$ and $\lambda_{i_t}$ a linear character of $H_{i_t}$ such that
							$\langle \lambda_{i_t}^A,\theta_{i_t} \rangle \neq 0 \text{ and } 
							\langle \lambda_{i_t}^A,\eta_{i_t+j_1+\ldots+j_{t-1}} \rangle = 0.$
							For $i\in \{1,\ldots,p\}\setminus \{i_1,\ldots,i_s\}$, we let $H_{i}$ to be the trivial subgroup of $A$
							and $\lambda_i = 1_{H_{i}}$, the trivial character of $H_i$.
							
							We let $H:=H_1\times \cdots \times H_p$ and $\lambda=\lambda_1\times \cdots \times \lambda_p$, which is a linear character of $H$.
							From the construction of $\lambda$, it is easy to check that $\langle \lambda^B,\theta \rangle \neq 0$,
							hence $\langle \lambda^W,\chi \rangle = \langle \lambda^B,\chi_B \rangle \neq 0$. Also,
              $ \langle \lambda^W,\phi \rangle = \langle \lambda^B,\phi_B \rangle = \sum_{\ell=1}^p \prod_{j=1}^p 
							\langle \lambda_j^A, \eta_{j+\ell-1} \rangle =0.$		
\item[(i.b)] Since $A$ is almost monomial, from Theorem \ref{22}(3) it follows that $B$ is almost monomial. Obviously, $\theta \neq \eta$,
             hence there exists a subgroup $H\leqslant B$ and a linear 
             character $\lambda$ of $H$ such that $\langle \lambda^B, \theta \rangle > 0$ and $\langle \lambda^B, \eta \rangle = 0$. 
						 Since $\langle \eta, \chi_B \rangle > 0$, it follows that
             $\langle \lambda^W,\chi \rangle = \langle \lambda^B,\chi_B \rangle > 0.$
             On the other hand, as $\phi_B=\eta$, it follows that $\langle \lambda^W,\phi \rangle = \langle \lambda^B,\eta \rangle = 0$.
\item[(ii.a)] We have $\theta=\theta_1\times \cdots \times \theta_1$ and $\eta=\eta_1\times \cdots \times \eta_p$.
              Obviously, $\theta\neq \eta$, hence we may assume that $\theta_1\neq \eta_1$. Let $H_1\leqslant A$ be a subgroup and 
							$\lambda_1$ a linear character of $H$ such that 
              $\langle \lambda_1^A, \theta_1 \rangle \neq 0$ and $\langle \lambda_1^A, \eta_1 \rangle = 0$. Let $H=H_1\times\cdots\times H_1$
							and $\lambda=\lambda_1\times \cdots \times \lambda_1$. We have that $\lambda$ is linear and
              $$\langle \lambda^W,\chi \rangle = \langle \lambda^B,\chi_B \rangle = 
              \langle \lambda_1^A\times \cdots \times \lambda_1^A, \theta\times \cdots \times \theta \rangle = 
							(\langle \lambda_1^A,\theta \rangle) ^ p \neq 0.$$
              On the other hand, since $\phi_B=\sum_{\sigma\in \langle (12\cdots p )\rangle} \eta_{\sigma(1)}\times \eta_{\sigma(2)}\times
							\cdots \times \eta_{\sigma(p)}$, it follows that
              $$\langle \lambda^W,\phi \rangle = \langle \lambda^B,\phi_B \rangle = p \prod_{j=1}^p \langle \lambda_1^A,\eta_j \rangle = 0.$$
\item[(ii.b)] If $\theta\neq \eta$ then, as $B$ is almost monomial, there exists a subgroup $H\leqslant B$ and a linear 
              character $\lambda$ of $H$ such that $\langle \lambda^B, \theta \rangle > 0$ and $\langle \lambda^B, \eta \rangle = 0$.
              We have that $$\langle \lambda^W,\chi \rangle = \langle \lambda^B,\chi_B \rangle =  \langle \lambda^B,\theta \rangle > 0 
							\text{ and }
              \langle \lambda^W,\phi \rangle = \langle \lambda^B,\phi_B \rangle =  \langle \lambda^B,\eta \rangle = 0.$$
              Now, assume $\chi_B=\phi_B=\theta=\theta_1\times \theta_1 \times \cdots \times \theta_1$. 
						  As $\chi\neq \phi$ are constituents of $\theta^W$,
              according to \cite[Corollary 6.17]{isaac} (Gallagher), it follows that there exists a nontrivial irreducible character 
						  $\beta$ of $W/B$
              such that $\phi=\beta \chi$, where $\beta$ is seen as a linear character of $W$ with $\Ker(\beta)=B$. From hypothesis,
						  there exists a subgroup $H\leqslant W$ and a linear character $\lambda$ of $H$ such that $\chi\in \Cons(\lambda^W)$
              and $\beta\chi \notin \Cons(\lambda^W)$. 
\end{enumerate}
\end{proof}

\begin{obs}\rm
It is natural to ask if, in Theorem \ref{main}, we can drop the technical hypothesis: "For any $\chi\in \Irr(W)$ with $\chi_B\in\Irr(B)$ and any non trivial character $\beta$ of $W/B$, 
there exists a subgroup $H\subset W$ and a linear character $\lambda$ of $H$ such that $\chi\in \Cons(\lambda^W)$
and $\beta\chi \notin \Cons(\lambda^W)$.". In other words, is it true, in general, that if $A$ is almost monomial and $C$ is a cyclic group of order prime $p>0$, then
$W=A\wr C$ is almost monomial? We couldn't find a counterexample using GAP \cite{gap}, but we believe that this assertion is false.
\end{obs}

\begin{obs}\rm
With the notations from Theorem \ref{main}, if $\chi\in \Irr(W)$ such
that $\theta=\chi_B\in \Irr(B)$, then $(\theta^G)_T=(\theta_{B\cap T})^T = \theta(1)\rho_T = \chi(1)\rho_T$, where $\rho_T=\nu_1+\cdots+\nu_p$ is the regular character of $T$,
and $\nu_1,\ldots,\nu_p$ are the irreducible (linear) characters of $T$. 
On the other hand, $\theta^G = \chi_1+\cdots+\chi_p$, where $\chi_j=\nu_j\chi$ for $1\leq j\leq p$, 
hence $(\chi_1)_T + \cdots + (\chi_p)_T = \chi(1)\rho_T$. However, this does not guarantee (and in fact is false!) that $|\Cons((\chi_j)_T)|=1$, so we cannot use the
subgroup $T\leqslant W$ in order to check the definition of almost monomial groups for $\chi_1, \ldots,\chi_p$. 

If we want to find a subgroup 
$H\subset W$ and a linear character $\lambda$ of $H$ such that let's say $\chi_1\in \Cons(\lambda^W)$ and $\chi_2 \notin \Cons(\lambda^W)$, 
we have to consider subgroups of the form $HT\subset W$, where $H=H_1\times \cdots \times H_1$ and 
$H_1$ is a subgroup of $A$, and linear characters $\lambda$ of $HT$.
\end{obs}

\begin{obs}\rm
We say that a group $A$ is called \emph{normally almost monomial} if  for any two irreducible characters $\chi\neq \phi$ of $G$,
there exist a normal subgroup $N\leqslant G$ and a linear character $\lambda$ of $N$, such that $\langle \chi, \lambda^G \rangle>0$ and
$\langle \phi, \lambda^G \rangle=0$; see \cite[Definition 5.1]{cim2}. Obviously, any normally almost monomial group is almost monomial, but the converse is not true;
see \cite[Remark 5.3]{cim2}.
We note that if $A$ is normally almost monomial, then $W=A\wr C$ may not be. 
For example, the dihedral group $A=D_{10}$  is normally almost monomial, but $W=D_{10}\wr \CC_2$ is not. 
However, $W$ is almost monomial.
\end{obs}

{}


\addcontentsline{toc}{section}{References}
\begin{thebibliography}{}
\bibitem{artin2} E.\ Artin, \emph{Zur Theorie der L-Reihen mit allgemeinen Gruppencharakteren}, Abh. Math. Sem. Hamburg \textbf{8} (1931), 292--306.
\bibitem{brauer} R.\ Brauer, \emph{On Artin's L-series with general group characters}, Ann. of Math. \textbf{48} (1947), 502--514.
\bibitem{cim2} M.\ Cimpoea\c s, \emph{On a generalization of monomial groups}, arxiv.org/pdf/1910.12683.pdf
\bibitem{forum} M.\ Cimpoea\c s, F.\ Nicolae, \emph{Independence of Artin L-functions},  Forum Math. \textbf{31, no.2} (2019), 529--534.
\bibitem{lucrare} M.\ Cimpoea\c s, F.\ Nicolae, \emph{Artin L-functions to almost monomial Galois groups},  Forum Math. \textbf{32 no. 4} (2020), 937--940.
\bibitem{dorn} L.\ Dornhoff, \text{M-groups and 2-groups}, Math. Zeit. \textbf{100} (1967), 226--256.
\bibitem{gap}  The GAP Group, \emph{GAP Groups, Algorithms, and Programming}, Version 4.10.2 (2019), https://www.gap-system.org.
\bibitem{isaac} I.\ M.\ Isaacs, \emph{Character Theory of Finite Groups}, Dover, NY, (1994).
\bibitem{isaac2} I.\ M.\ Isaacs, \emph{Characters of Solvable Groups}, Graduate Studies in Mathematics Volume 189 (2018).
\bibitem{konig} J.\ K\"onig, \emph{Solvability of Generalized Monomial Groups}, J. Group Theory \textbf{13} (2010), 207--220.
\bibitem{monat} F.\ Nicolae, \emph{On holomorphic Artin L-functions}, Monatsh. Math. \textbf{186, no. 4}, (2018), 679--683.
\bibitem{taketa} K.\ Taketa, \emph{\"Uber die Gruppen, deren Darstellungen sich s\"amtlich auf monomiale Gestalt transformieren lassen}, Proc. Acad. Tokyo 6 (1930), 31--33.
\end{thebibliography}
\end{document}